\documentclass{amsart}
\usepackage[nobysame]{amsrefs}
\usepackage{amssymb}
\usepackage[usesnames,svgnames]{xcolor}
\usepackage[sc]{mathpazo}
\usepackage{tikz}
\usepackage{enumerate}
\usepackage{subfigure}
\usepackage[colorlinks=true,
	urlcolor=MidnightBlue,
	citecolor=DarkGreen]{hyperref}
\newtheorem{theorem}{Theorem}[section]
\newtheorem{proposition}[theorem]{Proposition}
\newtheorem{conjecture}[theorem]{Conjecture}
\newtheorem{lemma}[theorem]{Lemma}
\newtheorem{corollary}[theorem]{Corollary}
\newtheorem{definition}[theorem]{Definition}
\theoremstyle{remark}
\newtheorem{remark}[theorem]{Remark}
\newtheorem{example}[theorem]{Example}
\newcommand{\alert}[1]{{\color{DarkGreen}\emph{#1}}}
\newcommand{\ie}{\text{i.e.}\;}
\newcommand{\PP}{\mathcal{P}}

\newcommand{\bruh}{\mathcal{B}_{\gamma}}
\author{Henri M{\"u}hle}
\address{LIAFA, Universit{\'e} Paris Diderot, Case 7014, F-75205 Paris Cedex 13, France}
\thanks{This work was funded by the FWF Research Grant No. Z130-N13, and by a Public Grant overseen by the French National Research Agency (ANR) as part of the ``Investissements d'Avenir'' Program (Reference: ANR-10-LABX-0098).}
\email{henri.muehle@liafa.univ-paris-diderot.fr}
\title{SB-Labelings, Distributivity, and Bruhat Order on Sortable Elements}
\keywords{SB-Labeling, M{\"o}bius function, Crosscut complex, Distributive lattice, Join-distributive lattice, Antimatroids, Bruhat order, Sortable Elements, Coxeter groups}
\subjclass[2010]{20F55 (primary), and 06D75, 06A07 (secondary)}

\begin{document}

\allowdisplaybreaks

\begin{abstract}
	In this article, we investigate the set of $\gamma$-sortable elements, associated with a Coxeter group $W$ and a Coxeter element $\gamma\in W$, under Bruhat order, and we denote this poset by $\mathcal{B}_{\gamma}$. We show that this poset belongs to the class of SB-lattices recently introduced by Hersh and M{\'e}sz{\'a}ros, by proving a more general statement, namely that all join-distributive lattices are SB-lattices. The observation that $\mathcal{B}_{\gamma}$ is join-distributive is due to Armstrong. Subsequently, we investigate for which finite Coxeter groups $W$ and which Coxeter elements $\gamma\in W$ the lattice $\mathcal{B}_{\gamma}$ is in fact distributive. It turns out that this is the case for the ``coincidental'' Coxeter groups, namely the groups $A_{n},B_{n},H_{3}$ and $I_{2}(k)$. We conclude this article with a conjectural characteriziation of the Coxeter elements $\gamma$ of said groups for which $\mathcal{B}_{\gamma}$ is distributive in terms of forbidden orientations of the Coxeter diagram. 
\end{abstract}

\maketitle

\section{Introduction}
	\label{sec:introduction}
Recently, Hersh and M{\'e}sz{\'a}ros introduced a new class of lattices, so-called SB-lattices~\cite{hersh14sb}. They showed that these lattices admit a certain edge-labeling, which implies that the order complex of every open interval of this lattice is homotopy equivalent to a sphere or a ball (hence the name). Equivalently, the M{\"o}bius function of such a lattice takes values only in $\{-1,0,1\}$. In the same paper they showed that every distributive lattice admits an SB-labeling, and they showed that the same is true for the weak order on a Coxeter group and for the Tamari lattice. We extend their results by showing that another class of lattices, so-called join-semidistributive lattice, belong to the class of SB-lattices as well.

Subsequently, we investigate a particular family of join-semidistributive lattices, namely the set of $\gamma$-sortable elements of a Coxeter groups equipped with Bruhat order. We denote this poset by $\bruh$. The $\gamma$-sortable elements were defined by Reading and Speyer, see \cites{reading07sortable,reading11sortable}. To summarize, the first main result of this paper is the following.

\begin{theorem}\label{thm:main_sb}
	Every join-distributive lattice is an SB-lattice. In particular, the Bruhat order on $\gamma$-sortable elements is an SB-lattice for every Coxeter group $W$ and every Coxeter element $\gamma\in W$. 
\end{theorem}

Join-distributivity is a generalization of distributivity, and while working with the poset $\bruh$ we observed that for some Coxeter groups and for some Coxeter elements the lattice $\bruh$ is in fact distributive. This led us to the question whether we can characterize the (finite) Coxeter groups $W$ and the Coxeter elements $\gamma\in W$ for which $\bruh$ is distributive. We approach this problem by looking for forbidden orientations of the Coxeter diagram of $W$, and we prove the following result.

\begin{theorem}\label{thm:main_distributive}
	Let $W$ be a finite Coxeter group. There exists a Coxeter element $\gamma\in W$ such that $\bruh$ is distributive if and only if $W$ is of type $A_{n},B_{n},H_{3}$ or $I_{2}(k)$. 
\end{theorem}

The finite Coxeter groups appearing in Theorem~\ref{thm:main_distributive} are sometimes called the ``coincidental types'', since these groups enjoy a list of properties that distinguishes them from the other finite (complex) reflection groups, see \cite{fomin05generalized}*{Theorems~8.5~and~10.2}, \cite{miller15foulkes}*{Theorem~14}, \cite{reading08chains}, and \cite{williams13cataland}*{Remark~3.1.26}. Theorem~\ref{thm:main_distributive} adds another property to this list. We conclude this article with a conjectural characterization of the Coxeter elements $\gamma\in W$ for which $\bruh$ is distributive.

\smallskip

This article is organized as follows: in Section~\ref{sec:preliminaries} we define SB-lattices, Coxeter groups and sortable elements. We subsequently prove Theorem~\ref{thm:main_sb} in Section~\ref{sec:join_distributive lattices}, and define the poset $\bruh$. We conclude this paper by proving Theorem~\ref{thm:main_distributive} in Section~\ref{sec:distributivity}.

\section{Preliminaries}
	\label{sec:preliminaries}
In this section, we recall the basic concepts needed in this article. For more background on SB-labelings, we refer to \cite{hersh14sb}. For any undefined notation and additional information on Coxeter groups and sortable elements, we refer to \cite{bjorner05combinatorics} and \cite{reading11sortable}, respectively.

\subsection{SB-Labelings}
	\label{sec:sb_labelings}
Let $\PP=(P,\leq)$ be a (possibly infinite) poset. An element $p\in P$ is \alert{covered by} another element $q\in P$ (denoted by $p\lessdot q$) if $p<q$ and there exists no $z\in P$ with $p<z<q$. Accordingly, $q$ \alert{covers} $p$ and the elements $p$ and $q$ form a \alert{cover relation} or an \alert{edge} in $\PP$. The set $\mathcal{E}(\PP)=\bigl\{(p,q)\mid p\lessdot q\bigr\}$ is the \alert{Hasse diagram} of $\PP$. 

For $p\leq q$ we call a set of the form $[p,q]=\{z\in P\mid p\leq z\leq q\}$ a \alert{(closed) interval} of $\PP$. A \alert{saturated chain} in an interval $[p,q]$ is a sequence $(p,z_{1},z_{2},\ldots,z_{k-1},q)$, where $z_{i}\in P$ for $i\in\{1,2,\ldots,k-1\}$ and $p\lessdot z_{1}\lessdot z_{2}\lessdot\cdots\lessdot z_{k-1}\lessdot q$. 

A poset $\PP$ is a \alert{lattice} if any two elements in $\PP$ have a least upper bound (a \alert{join}) and a greatest lower bound (a \alert{meet}), denoted by $\vee$ and $\wedge$, respectively. A lattice is \alert{locally finite} if every interval is finite.

An \alert{edge-labeling} of $\PP$ is a map $\lambda:\mathcal{E}(\PP)\to\Lambda$, for some set of labels $\Lambda$. An \alert{SB-labeling} of a lattice $\PP$ is an edge-labeling $\lambda$ of $\PP$ that satisfies the following properties for every $p,p_{1},p_{2}\in P$ with $p\lessdot p_{1},p_{2}$:
\begin{enumerate}[(i)]
	\item $\lambda(p,p_{1})\neq\lambda(p,p_{2})$;
	\item each saturated chain in the interval $[p,p_{1}\vee p_{2}]$ uses both labels $\lambda(p,p_{1})$ and $\lambda(p,p_{2})$ a positive number of times;\quad and
	\item none of the saturated chains in the interval $[p,p_{1}\vee p_{2}]$ uses any other label besides $\lambda(p,p_{1})$ and $\lambda(p,p_{2})$.
\end{enumerate}

A locally finite lattice with a least element that admits an SB-labeling is called an \alert{SB-lattice}.

\begin{remark}
	In fact, the original definition of an SB-labeling given in \cite{hersh14sb}*{Definition~3.2} was phrased a bit differently, but it was shown in \cite{hersh14sb}*{Theorem~3.5} that the above definition is equivalent to the original definition. 
\end{remark}

SB-lattices enjoy the following nice property.

\begin{theorem}[\cite{hersh14sb}*{Theorem~3.8}]\label{thm:sb_mobius}
	The M{\"o}bius function of an SB-lattice takes values only in $\{-1,0,1\}$. 
\end{theorem}

\subsection{Coxeter Groups}
  \label{sec:coxeter_groups}
A \alert{Coxeter group} is a group $W$ admitting a presentation
\begin{displaymath}
	W = \bigl\langle s_{1},s_{2},\ldots,s_{n}\mid (s_{i}s_{j})^{m_{i,j}}=\varepsilon,\;\text{for}\;i,j\in\{1,2,\ldots,n\}\bigr\rangle,
\end{displaymath}	
where $\varepsilon\in W$ is the identity and the numbers $m_{i,j}$ are either positive integers or the formal symbol $\infty$ for all $i,j\in\{1,2,\ldots,n\}$ such that $m_{i,j}\geq 2$ if $i\neq j$, and $m_{i,i}=1$. (We use the convention that $\infty$ is formally larger than every integer.) The elements in $S=\{s_{1},s_{2},\ldots,s_{n}\}$ are the \alert{Coxeter generators} of $W$, and $n$ is the \alert{rank} of $W$. A subgroup of $W$ that is generated by a subset $J\subseteq S$ is a Coxeter group in its own right and is called a \alert{standard parabolic subgroup} of $W$. A Coxeter group is called \alert{irreducible} if it is not isomorphic to a direct product of Coxeter groups of smaller rank. The finite irreducible Coxeter groups were completely classified by Coxeter in \cite{coxeter35complete}. This classification is best visualized using so-called Coxeter diagrams. The \alert{Coxeter diagram} of $W$, denoted by $\Gamma(W)$, is a labeled graph whose vertices are the Coxeter generators of $W$, and two generators $s_{i}$ and $s_{j}$ are connected by an edge if and only if $m_{i,j}\geq 3$. In addition, the edge between $s_{i}$ and $s_{j}$ is labeled by $m_{i,j}$ if $m_{i,j}\geq 4$. It is not hard to see that $W$ is irreducible if and only if its Coxeter diagram is connected. Figure~\ref{fig:coxeter_diagrams} shows the Coxeter diagrams of the finite irreducible Coxeter groups. 

\begin{figure}
	\centering
	\begin{tabular}{c|l}
		Name & Coxeter diagram\\
		\hline
		$A_{n},\quad n\geq 1$ & 
		\raisebox{-.2cm}{\begin{tikzpicture}\small
			\draw(4,.2) node{};
			\draw(4,-.2) node{};
			\draw(0,0) node(s1){$s_{1}$};
			\draw(1.5,0) node(s2){$s_{2}$};
		\draw(3,0) node(s3){$s_{3}$};
			\draw(3.67,0) node(sd){$\cdots$};
			\draw(4.5,0) node(sn1){$s_{n-1}$};
			\draw(6,0) node(sn){$s_{n}$};
			\draw[thick](s1) -- (s2) -- (s3) -- (sd) -- (sn1) -- (sn);
		\end{tikzpicture}}\\
		$B_{n},\quad n\geq 2$ & 
		\raisebox{-.2cm}{\begin{tikzpicture}\small
			\draw(4,.2) node{};
			\draw(4,-.2) node{};
			\draw(0,0) node(s1){$s_{1}$};
			\draw(1.5,0) node(s2){$s_{2}$};
			\draw(3,0) node(s3){$s_{3}$};
			\draw(3.67,0) node(sd){$\cdots$};
			\draw(4.5,0) node(sn1){$s_{n-1}$};
			\draw(6,0) node(sn){$s_{n}$};
			\draw(5.33,.15) node{\tiny $4$};
			\draw[thick](s1) -- (s2) -- (s3) -- (sd) -- (sn1) -- (sn);
		\end{tikzpicture}}\\
		$D_{n},\quad n\geq 4$ & 
		\raisebox{-.2cm}{\begin{tikzpicture}\small
			\draw(4,-.2) node{};
			\draw(0,0) node(s1){$s_{1}$};
			\draw(1.5,0) node(s2){$s_{2}$};
			\draw(3,0) node(s3){$s_{3}$};
			\draw(3.67,0) node(sd){$\cdots$};
			\draw(4.5,0) node(sn2){$s_{n-2}$};
			\draw(4.5,.75) node(sn1){$s_{n-1}$};
			\draw(6,0) node(sn){$s_{n}$};
			\draw[thick](s1) -- (s2) -- (s3) -- (sd) -- (sn2) -- (sn1);
			\draw[thick](sn2) -- (sn);
		\end{tikzpicture}}\\
		$E_{6}$ & 
		\raisebox{-.2cm}{\begin{tikzpicture}\small
			\draw(4,-.2) node{};
			\draw(0,0) node(s1){$s_{1}$};
			\draw(1.5,0) node(s2){$s_{2}$};
			\draw(3,0) node(s3){$s_{3}$};
			\draw(3,.75) node(s6){$s_{6}$};
			\draw(4.5,0) node(s4){$s_{4}$};
			\draw(6,0) node(s5){$s_{5}$};
			\draw[thick](s1) -- (s2) -- (s3) -- (s4) -- (s5);
			\draw[thick](s3) -- (s6);
		\end{tikzpicture}}\\
		$E_{7}$ & 
		\raisebox{-.2cm}{\begin{tikzpicture}\small
			\draw(4,-.2) node{};
			\draw(0,0) node(s1){$s_{1}$};
			\draw(1.5,0) node(s2){$s_{2}$};
			\draw(3,0) node(s3){$s_{3}$};
			\draw(3,.75) node(s7){$s_{7}$};
			\draw(4.5,0) node(s4){$s_{4}$};
			\draw(6,0) node(s5){$s_{5}$};
			\draw(7.5,0) node(s6){$s_{6}$};
			\draw[thick](s1) -- (s2) -- (s3) -- (s4) -- (s5) -- (s6);
			\draw[thick](s3) -- (s7);
		\end{tikzpicture}}\\
		$E_{8}$ & 
		\raisebox{-.2cm}{\begin{tikzpicture}\small
			\draw(4,-.2) node{};
			\draw(0,0) node(s1){$s_{1}$};
			\draw(1.5,0) node(s2){$s_{2}$};
			\draw(3,0) node(s3){$s_{3}$};
			\draw(3,.75) node(s8){$s_{8}$};
			\draw(4.5,0) node(s4){$s_{4}$};
			\draw(6,0) node(s5){$s_{5}$};
			\draw(7.5,0) node(s6){$s_{6}$};
			\draw(9,0) node(s7){$s_{7}$};
			\draw[thick](s1) -- (s2) -- (s3) -- (s4) -- (s5) -- (s6) -- (s7);
			\draw[thick](s3) -- (s8);
		\end{tikzpicture}}\\
		$F_{4}$ & 
		\raisebox{-.2cm}{\begin{tikzpicture}\small
			\draw(4,.2) node{};
			\draw(4,-.2) node{};
			\draw(0,0) node(s1){$s_{1}$};
			\draw(1.5,0) node(s2){$s_{2}$};
			\draw(3,0) node(s3){$s_{3}$};
			\draw(4.5,0) node(s4){$s_{4}$};
			\draw(2.25,.15) node{\tiny $4$};
			\draw[thick](s1) -- (s2) -- (s3) -- (s4);
		\end{tikzpicture}}\\
		$H_{3}$ & 
		\raisebox{-.2cm}{\begin{tikzpicture}\small
			\draw(4,.2) node{};
			\draw(4,-.2) node{};
			\draw(0,0) node(s1){$s_{1}$};
			\draw(1.5,0) node(s2){$s_{2}$};
			\draw(3,0) node(s3){$s_{3}$};
			\draw(2.25,.15) node{\tiny $5$};
			\draw[thick](s1) -- (s2) -- (s3);
		\end{tikzpicture}}\\
		$H_{4}$ & 
		\raisebox{-.2cm}{\begin{tikzpicture}\small
			\draw(4,.2) node{};
			\draw(4,-.2) node{};
			\draw(0,0) node(s1){$s_{1}$};
			\draw(1.5,0) node(s2){$s_{2}$};
			\draw(3,0) node(s3){$s_{3}$};
			\draw(4.5,0) node(s4){$s_{4}$};
			\draw(3.75,.15) node{\tiny $5$};
			\draw[thick](s1) -- (s2) -- (s3) -- (s4);
		\end{tikzpicture}}\\
		$I_{2}(k),\quad k\geq 5$ & 
		\raisebox{-.2cm}{\begin{tikzpicture}\small
			\draw(4,.2) node{};
			\draw(4,-.2) node{};
			\draw(0,0) node(s1){$s_{1}$};
			\draw(1.5,0) node(s2){$s_{2}$};
			\draw(.75,.15) node{\tiny $k$};
			\draw[thick](s1) -- (s2);
		\end{tikzpicture}}\\
	\end{tabular}
	\caption{The Coxeter diagrams of the finite irreducible Coxeter groups.}
	\label{fig:coxeter_diagrams}
\end{figure}

Since $S$ is a generating set of $W$, we can write every $w\in W$ as a product of Coxeter generators. The least number of generators needed to form $w$, is called the \alert{Coxeter length} of $w$, and will be written as $\ell_{S}(w)$. We say that a word $w=s_{i_{1}}s_{i_{2}}\cdots s_{i_{k}}$ is \alert{reduced} if $\ell_{S}(w)=k$.

\subsection{Sortable Elements}
  \label{sec:sortable_elements}
Let $W$ be a Coxeter group of rank $n$. An element $\gamma\in W$ is called a \alert{Coxeter element} of $W$ if $\gamma=s_{\pi(1)}s_{\pi(2)}\cdots s_{\pi(n)}$ for some permutation $\pi$ of $\{1,2,\ldots,n\}$. Without loss of generality, we can restrict our attention to the Coxeter element $\gamma=s_{1}s_{2}\cdots s_{n}$. Consider the half-infinite word
\begin{displaymath}
	\gamma^{\infty}=s_{1}s_{2}\cdots s_{n}\vert s_{1}s_{2}\cdots s_{n}\vert s_{1}\cdots.
\end{displaymath}
The vertical bars have no influence on the structure of the word, but shall serve for a better readability. Clearly, for every $w\in W$, every reduced word for $w$ can be written as a subword of $\gamma^{\infty}$. We call the lexicographically first subword of $\gamma^{\infty}$ that is a reduced word for $w$, the \alert{$\gamma$-sorting word} of $w$, and we denote it by $\gamma(w)$. We can write
\begin{displaymath}
	\gamma(w) = s_{1}^{\delta_{1,1}}s_{2}^{\delta_{1,2}}\cdots s_{n}^{\delta_{1,n}}\vert s_{1}^{\delta_{2,1}}s_{2}^{\delta_{2,2}}\cdots s_{n}^{\delta_{2,n}}\vert\cdots\vert s_{1}^{\delta_{l,1}}s_{2}^{\delta_{l,2}}\cdots s_{n}^{\delta_{l,n}},
\end{displaymath}
for $l\in\mathbb{N}$ and $\delta_{i,j}\in\{0,1\}$ for $i\in\{1,2,\ldots,l\}$ and $j\in\{1,2,\ldots,n\}$. The \alert{$i$-th block} of $w$ is the set $b_{i}=\{s_{j}\mid\delta_{i,j}=1\}$. We say that $w$ is \alert{$\gamma$-sortable} if $b_{1}\supseteq b_{2}\supseteq\cdots\supseteq b_{l}$, and we write $C_{\gamma}$ for the set of $\gamma$-sortable elements of $W$. Further, define the set of filled positions of $w$ by
\begin{displaymath}
	\alpha_{\gamma}(w)=\bigl\{(i-1)n+j\mid \delta_{i,j}=1\bigr\}. 
\end{displaymath}
We notice that $\alpha_{\gamma}$ depends on the choice of reduced word for $\gamma$, while $C_{\gamma}$ does not.

\begin{remark}\label{rem:sortable_origin}
	The concept of $\gamma$-sortability was introduced by Reading in \cite{reading07sortable} as a generalization of stack-sortability, and was used to define the family of Cambrian lattices associated with a Coxeter group \cites{reading06cambrian,reading11sortable}. The number of $\gamma$-sortable elements of a finite Coxeter group $W$ is the $W$-Catalan number, defined in \cite{bessis03dual}*{Section~5.2}, and Reading used the $\gamma$-sortable elements to provide a bridge between the noncrossing partitions of $W$ and the clusters of $W$, see \cite{reading07clusters}.
	
	The concept of $\gamma$-sortability has been further extended by Armstrong in \cite{armstrong09sorting}, where he defined $\omega$-sortability for an arbitrary, not necessarily reduced word $\omega$ in the Coxeter generators of $W$, and if $\omega=\gamma^{\infty}$, then one obtains precisely the $\gamma$-sortable elements.
\end{remark}

\section{Join-Distributive Lattices}
	\label{sec:join_distributive lattices}
In this section we prove the first part of Theorem~\ref{thm:main_sb}, namely that every join-distributive lattice admits an SB-labeling, see Theorem~\ref{thm:join_distributive_sb} below. Let us first recall the necessary definitions. A lattice $\PP=(P,\leq)$ is \alert{meet-semidistributive} if for every three elements $p,q,r\in P$ with $p\wedge q=p\wedge r$ we have $p\wedge q=p\wedge(q\vee r)$. Moreover, $\PP$ is \alert{upper semimodular} if for every two elements $p,q\in P$ with $p\wedge q\lessdot p,q$, we have $p,q\lessdot p\vee q$. Then, a lattice is \alert{join-distributive} if it is both meet-semidistributive and upper semimodular. See \cite{edelman80meet} for more information on join-distributive lattices. 

Recall further that an \alert{antimatroid} is a pair $(M,\mathcal{F})$, where $M$ is a set and $\mathcal{F}\subseteq\wp(M)$ is a family of subsets of $M$ that satisfies the following properties: 
\begin{enumerate}[(i)]
	\item $\emptyset\in\mathcal{F}$; \quad and
	\item if $X,Y\in\mathcal{F}$ with $Y\not\subseteq X$, then there exists some $x\in X\setminus Y$ such that $X\cup\{x\}\in\mathcal{F}$. 
\end{enumerate}
The elements of $\mathcal{F}$ are called the \alert{feasible sets} of $(M,\mathcal{F})$. We have the following result. 

\begin{theorem}[\cite{edelman80meet}*{Theorem~3.3}]\label{thm:join_distributive_antimatroid}
	A lattice $\PP$ is join-distributive if and only if there exists an antimatroid $(M,\mathcal{F})$ such that $\PP\cong(\mathcal{F},\subseteq)$. 
\end{theorem}

In view of this correspondence we can now conclude the following result.

\begin{theorem}\label{thm:join_distributive_sb}
	Every join-distributive lattice admits an SB-labeling.
\end{theorem}
\begin{proof}
	Let $\PP$ be a join-distributive lattice. In view of Theorem~\ref{thm:join_distributive_antimatroid}, we can view $\PP$ as a lattice of feasible sets of some antimatroid $(M,\mathcal{F})$, and thus every edge in $\PP$ is determined by a pair $X,Y\in\mathcal{F}$ with $Y\setminus X=\{x\}$. This induces an edge-labeling of $\PP$, which we will denote by $\lambda_{\mathcal{F}}$. Since $\PP$ is upper-semimodular, it follows that for any $p,p_{1},p_{2}\in P$ with $p\lessdot p_{1},p_{2}$, where we write $\bar{p}=p_{1}\vee p_{2}$, we have $p_{1},p_{2}\lessdot\bar{p}$. Since $\PP$ is meet-semidistributive, it follows that the interval $[p,\bar{p}]$ consists only of the four elements $p,p_{1},p_{2},\bar{p}$. Thus we have $\lambda_{\mathcal{F}}(p,p_{1})=\lambda_{\mathcal{F}}(p_{2},\bar{p})$ and $\lambda_{\mathcal{F}}(p,p_{2})=\lambda_{\mathcal{F}}(p_{1},\bar{p})$, as well as $\lambda_{\mathcal{F}}(p,p_{1})\neq\lambda_{\mathcal{F}}(p,p_{2})$. Hence $\lambda_{\mathcal{F}}$ is an SB-labeling of $\PP$.	
\end{proof}

\begin{corollary}\label{cor:join_semidistributive_mobius}
	The M{\"o}bius function of a join-semidistributive lattice takes values only in $\{-1,0,1\}$.
\end{corollary}
\begin{proof}
	This follows from Theorems~\ref{thm:sb_mobius} and \ref{thm:join_distributive_sb}.
\end{proof}

\begin{remark}
	Join-distributivity can be seen as a generalization of distributivity, see \eqref{eq:meet_distributivity} and \eqref{eq:join_distributivity} below. In that sense, Theorem~\ref{thm:join_distributive_sb} generalizes \cite{hersh14sb}*{Theorem~5.1}, which states that every distributive lattice is an SB-lattice.
\end{remark}

All join-distributive lattices are by definition meet-semidistributive. Lattices that satisfy the meet-distributive law and the corresponding dual law are called \alert{semidistributive}. Obviously, every distributive lattice is also semidistributive, but semidistributive lattices need no longer be graded. It is known that the M{\"o}bius function of a semidistributive lattice takes values only in $\{-1,0,1\}$~\cite{farley13bijection}, and it would be interesting whether such lattices are always SB-lattices. 

Another generalization of distributivity to ungraded lattices, so-called \alert{trimness}, was introduced by Thomas in \cite{thomas06analogue}. It is the statement of \cite{thomas06analogue}*{Theorem~7} that the M{\"o}bius function of a trim lattice also takes values only in $\{-1,0,1\}$, and again it would be interesting to know whether trim lattices are always SB-lattices. 

An important example of lattices that belong to both previously mentioned classes of lattices are Reading's Cambrian semilattices, see \cite{reading11sortable}*{Theorem~8.1} and \cite{muehle15trimness}*{Theorem~1.1}. The Cambrian semilattices generalize the Tamari lattices to all Coxeter groups. Theorem~5.5 in \cite{hersh14sb} states that the Tamari lattices are SB-lattices. We could produce SB-labelings for some small Cambrian semilattices, but we could not find a uniform definition of such a labeling. We nevertheless pose the following conjecture.

\begin{conjecture}
	Let $W$ be a Coxeter group, and let $\gamma\in W$ be a Coxeter element. The $\gamma$-Cambrian semilattice, \ie the set $C_{\gamma}$ equipped with the weak order on $W$, is an SB-lattice.
\end{conjecture}

\begin{remark}
	Recently, McConville has investigated a slightly weaker lattice property, namely \alert{crosscut-simpliciality} \cite{mcconville14crosscut}. It follows by definition that every SB-lattice is crosscut-simplicial. He showed in particular that meet-semidistributive lattices are crosscut-simplicial, see \cite{mcconville14crosscut}*{Theorem~1.2}. 
\end{remark}
		
\subsection{The Bruhat Order on Sortable Elements}
	\label{sec:bruhat}
In this section, we consider a special family of join-distributive lattices, namely the set of $\gamma$-sortable elements of a Coxeter group $W$ equipped with the Bruhat order, which we define next.

\begin{definition}
	For $u,v\in W$ we write $u\leq_{B}v$ if and only if there exists a reduced word $v=a_{1}a_{2}\cdots a_{l}$ and indices $1\leq i_{1}<i_{2}<\cdots<i_{k}\leq l$ such that $u=a_{i_{1}}a_{i_{2}}\cdots a_{i_{k}}$. The partial order $\leq_{B}$ is called the \alert{Bruhat order} on $W$.
\end{definition}

Clearly the identity $\varepsilon$ is the least element with respect to $\leq_{B}$. Moreover, the poset $(W,\leq_{B})$ is graded by $\ell_{S}$, but it is in general not a lattice. Thus we restrict our attention to the subposet $\bruh=(C_{\gamma},\leq_{B})$, and in what follows, we index poset-theoretic notions that refer to the Bruhat order on $\gamma$-sortable elements by ``B'', \ie an interval in the poset $\bruh$ will be denoted by $[u,v]_{B}$, and likewise for joins and meets. Recall that a lattice is \alert{finitary} if every principal order ideal is finite. In particular, finitary lattices are locally finite. 

\begin{theorem}\label{thm:bruhat_lattice}
	The poset $\bruh$ is a finitary lattice for every Coxeter group $W$ and every Coxeter element $\gamma\in W$. 
\end{theorem}
\begin{proof}
	First of all, let $w\in C_{\gamma}$ with $\ell_{S}(w)=k$. The interval $[\varepsilon,w]_{B}$ is certainly finite, since $w$ has finite length. Moreover, using the terminology from above, it is easy to see that $w\leq_{B}w'$ if and only if $\alpha_{\gamma}(w)\subseteq\alpha_{\gamma}(w')$. 
	
	Let $u,u'\in C_{\gamma}$. The word $\bar{u}$ defined by $\alpha_{\gamma}(\bar{u})=\alpha_{\gamma}(u)\cup\alpha_{\gamma}(u')$ is again $\gamma$-sortable. In particular, $\bar{u}$ is the least upper bound for both $u$ and $u'$. Hence the interval $[\varepsilon,\bar{u}]_{B}$ is finite, and analogously to before we see that any two elements in this interval possess a join. Hence it is a classical lattice-theoretic result that $[\varepsilon,\bar{u}]_{B}$ is a lattice. It follows immediately that the meet of $u$ and $u'$ exists as well, and the proof is complete.
\end{proof}

Theorem~\ref{thm:bruhat_lattice} was already mentioned in \cite{armstrong09sorting}*{Section~6}. It should be remarked
that in general $\bruh$ is an infinite lattice with no greatest element, which implies in particular that $\bruh$ is no complete lattice. The following result is also implicit in \cite{armstrong09sorting}.

\begin{theorem}[\cite{armstrong09sorting}]\label{thm:bruhat_join_semidistributive}
	The lattice $\bruh$ is join-distributive for every Coxeter group $W$ and every Coxeter element $\gamma\in W$. 
\end{theorem}

Hence the second part of Theorem~\ref{thm:main_sb} follows immediately from Theorem~\ref{thm:join_distributive_sb}. 

\begin{corollary}\label{cor:bruhat_sb}
	The lattice $\bruh$ admits an SB-labeling for every Coxeter group $W$ and every Coxeter element $\gamma\in W$.
\end{corollary}
\begin{proof}
	This follows from Theorems~\ref{thm:join_distributive_sb} and \ref{thm:bruhat_join_semidistributive}.
\end{proof}

\begin{proof}[Proof of Theorem~\ref{thm:main_sb}]
	This follows from Theorem~\ref{thm:join_distributive_sb} and Corollary~\ref{cor:bruhat_sb}.
\end{proof}

\begin{corollary}\label{cor:bruhat_mobius}
	The M{\"o}bius function of $\bruh$ takes values only in $\{-1,0,1\}$ for every Coxeter group $W$ and every Coxeter element $\gamma\in W$.
\end{corollary}

The correspondence between join-distributive lattices and antimatroids allows us to associate an antimatroid with $\bruh$. In general, this antimatroid is infinite, and thus the set of labels of the SB-labeling defined in the proof of Theorem~\ref{thm:join_distributive_sb} is potentially infinite. We notice that this labeling can be defined globally by
\begin{displaymath}
	\lambda_{\gamma}:\mathcal{E}(\bruh)\to\mathbb{N},\quad (u,v)\mapsto\alpha_{\gamma}(v)\setminus\alpha_{\gamma}(u).
\end{displaymath}
An analogous labeling was used in \cite{kallipoliti13on} to prove topological properties of the Cambrian semilattices. Now consider the map
\begin{displaymath}
	\eta:\mathbb{N}\to S,\quad i\mapsto \begin{cases}s_{n}, & \text{if}\;i\equiv 0\pmod{n},\\ s_{i\bmod{n}}, & \text{otherwise}.\end{cases}
\end{displaymath}
If we concatenate these two maps, then we obtain another SB-labeling of $\bruh$, which in contrast to $\lambda_{\gamma}$ uses only a finite set of labels:
\begin{align}\label{eq:sb_bruhat}
	b_{\gamma}:\mathcal{E}(\bruh)\to S,\quad (u,v)\mapsto\eta\bigl(\alpha_{\gamma}(v)\setminus\alpha_{\gamma}(u)\bigr).
\end{align}

\begin{example}
	Figure~\ref{fig:bruhat_sortable} shows two Bruhat lattices labeled by the SB-labeling defined in \eqref{eq:sb_bruhat}. The lattice in Figure~\ref{fig:bruhat_sortable_a3} is associated with the Coxeter group $A_{3}$ and the Coxeter element given by the oriented Coxeter diagram
	\raisebox{-.16cm}{\begin{tikzpicture}\small
		\draw(0,0) node(v1){$s_{1}$};
		\draw(1,0) node(v2){$s_{2}$};
		\draw(2,0) node(v3){$s_{3}$};
		\draw[thick,->](v1) -- (.55,0);
		\draw[thick](.55,0) -- (v2);
		\draw[thick,->](v2) -- (1.55,0);
		\draw[thick](1.55,0) -- (v3);
	\end{tikzpicture}}. Figure~\ref{fig:bruhat_sortable_c2} shows the first seven ranks of the lattice associated with the affine Coxeter group $\tilde{C}_{2}$ subject to the Coxeter element given by the oriented Coxeter diagram
		\raisebox{-.16cm}{\begin{tikzpicture}\small
		\draw(0,0) node(v1){$s_{1}$};
		\draw(.5,.2) node{\tiny 4};
		\draw(1,0) node(v2){$s_{2}$};
		\draw(1.5,.2) node{\tiny 4};
		\draw(2,0) node(v3){$s_{3}$};
		\draw[thick,->](v1) -- (.55,0);
		\draw[thick](.55,0) -- (v2);
		\draw[thick,->](v3) -- (1.45,0);
		\draw[thick](1.45,0) -- (v2);
	\end{tikzpicture}}.
	(See Section~\ref{sec:distributivity} for an explanation of the connection between Coxeter elements and orientations of the Coxeter diagram.)
	
	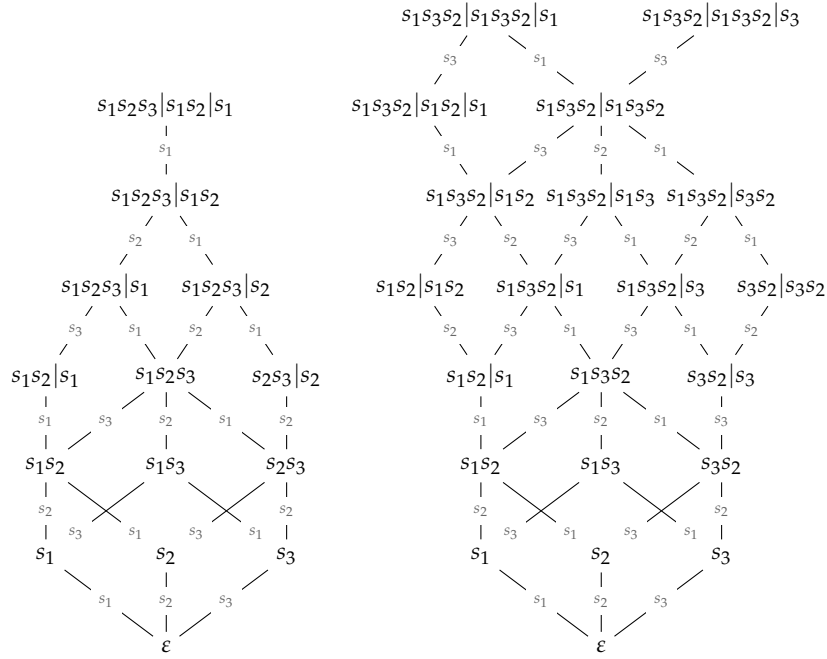
\begin{figure}
		\subfigure[$\mathcal{B}_{s_{1}s_{2}s_{3}}$ associated with $A_{3}$.]{\label{fig:bruhat_sortable_a3}
			\begin{tikzpicture}\small
				\def\x{1.6};
				\def\y{1.2};
				\draw(2*\x,1*\y) node(n1){$\varepsilon$};
				\draw(1*\x,2*\y) node(n2){$s_{1}$};
				\draw(2*\x,2*\y) node(n3){$s_{2}$};
				\draw(3*\x,2*\y) node(n4){$s_{3}$};
				\draw(1*\x,3*\y) node(n5){$s_{1}s_{2}$};
				\draw(2*\x,3*\y) node(n6){$s_{1}s_{3}$};
				\draw(3*\x,3*\y) node(n7){$s_{2}s_{3}$};
				\draw(1*\x,4*\y) node(n8){$s_{1}s_{2}\vert s_{1}$};
				\draw(2*\x,4*\y) node(n9){$s_{1}s_{2}s_{3}$};
				\draw(3*\x,4*\y) node(n10){$s_{2}s_{3}\vert s_{2}$};
				\draw(1.5*\x,5*\y) node(n11){$s_{1}s_{2}s_{3}\vert s_{1}$};
				\draw(2.5*\x,5*\y) node(n12){$s_{1}s_{2}s_{3}\vert s_{2}$};
				\draw(2*\x,6*\y) node(n13){$s_{1}s_{2}s_{3}\vert s_{1}s_{2}$};
				\draw(2*\x,7*\y) node(n14){$s_{1}s_{2}s_{3}\vert s_{1}s_{2}\vert s_{1}$};
				\draw(n1) -- (n2) node[fill=white] at(1.5*\x,1.5*\y) {\tiny\color{gray!80!black}$s_{1}$};
				\draw(n1) -- (n3) node[fill=white] at(2*\x,1.5*\y) {\tiny\color{gray!80!black}$s_{2}$};
				\draw(n1) -- (n4) node[fill=white] at(2.5*\x,1.5*\y) {\tiny\color{gray!80!black}$s_{3}$};
				\draw(n2) -- (n5) node[fill=white] at(1*\x,2.5*\y) {\tiny\color{gray!80!black}$s_{2}$};
				\draw(n2) -- (n6) node[fill=white] at(1.25*\x,2.25*\y) {\tiny\color{gray!80!black}$s_{3}$};
				\draw(n3) -- (n5) node[fill=white] at(1.75*\x,2.25*\y) {\tiny\color{gray!80!black}$s_{1}$};
				\draw(n3) -- (n7) node[fill=white] at(2.25*\x,2.25*\y) {\tiny\color{gray!80!black}$s_{3}$};
				\draw(n4) -- (n6) node[fill=white] at(2.75*\x,2.25*\y) {\tiny\color{gray!80!black}$s_{1}$};
				\draw(n4) -- (n7) node[fill=white] at(3*\x,2.5*\y) {\tiny\color{gray!80!black}$s_{2}$};
				\draw(n5) -- (n8) node[fill=white] at(1*\x,3.5*\y) {\tiny\color{gray!80!black}$s_{1}$};
				\draw(n5) -- (n9) node[fill=white] at(1.5*\x,3.5*\y) {\tiny\color{gray!80!black}$s_{3}$};
				\draw(n6) -- (n9) node[fill=white] at(2*\x,3.5*\y) {\tiny\color{gray!80!black}$s_{2}$};
				\draw(n7) -- (n9) node[fill=white] at(2.5*\x,3.5*\y) {\tiny\color{gray!80!black}$s_{1}$};
				\draw(n7) -- (n10) node[fill=white] at(3*\x,3.5*\y) {\tiny\color{gray!80!black}$s_{2}$};
				\draw(n8) -- (n11) node[fill=white] at(1.25*\x,4.5*\y) {\tiny\color{gray!80!black}$s_{3}$};
				\draw(n9) -- (n11) node[fill=white] at(1.75*\x,4.5*\y) {\tiny\color{gray!80!black}$s_{1}$};
				\draw(n9) -- (n12) node[fill=white] at(2.25*\x,4.5*\y) {\tiny\color{gray!80!black}$s_{2}$};
				\draw(n10) -- (n12) node[fill=white] at(2.75*\x,4.5*\y) {\tiny\color{gray!80!black}$s_{1}$};
				\draw(n11) -- (n13) node[fill=white] at(1.75*\x,5.5*\y) {\tiny\color{gray!80!black}$s_{2}$};
				\draw(n12) -- (n13) node[fill=white] at(2.25*\x,5.5*\y) {\tiny\color{gray!80!black}$s_{1}$};
				\draw(n13) -- (n14) node[fill=white] at(2*\x,6.5*\y) {\tiny\color{gray!80!black}$s_{1}$};
			\end{tikzpicture}}
		\subfigure[The first seven ranks of $\mathcal{B}_{s_{1}s_{3}s_{2}}$ associated with $\tilde{C}_{2}$.]{\label{fig:bruhat_sortable_c2}
			\begin{tikzpicture}\small
				\def\x{1.6};
				\def\y{1.2};
				\draw(2*\x,1*\y) node(n1){$\varepsilon$};
				\draw(1*\x,2*\y) node(n2){$s_{1}$};
				\draw(2*\x,2*\y) node(n3){$s_{2}$};
				\draw(3*\x,2*\y) node(n4){$s_{3}$};
				\draw(1*\x,3*\y) node(n5){$s_{1}s_{2}$};
				\draw(2*\x,3*\y) node(n6){$s_{1}s_{3}$};
				\draw(3*\x,3*\y) node(n7){$s_{3}s_{2}$};
				\draw(1*\x,4*\y) node(n8){$s_{1}s_{2}\vert s_{1}$};
				\draw(2*\x,4*\y) node(n9){$s_{1}s_{3}s_{2}$};
				\draw(3*\x,4*\y) node(n10){$s_{3}s_{2}\vert s_{3}$};
				\draw(.5*\x,5*\y) node(n11){$s_{1}s_{2}\vert s_{1}s_{2}$};
				\draw(1.5*\x,5*\y) node(n12){$s_{1}s_{3}s_{2}\vert s_{1}$};
				\draw(2.5*\x,5*\y) node(n13){$s_{1}s_{3}s_{2}\vert s_{3}$};
				\draw(3.5*\x,5*\y) node(n14){$s_{3}s_{2}\vert s_{3}s_{2}$};
				\draw(1*\x,6*\y) node(n15){$s_{1}s_{3}s_{2}\vert s_{1}s_{2}$};
				\draw(2*\x,6*\y) node(n16){$s_{1}s_{3}s_{2}\vert s_{1}s_{3}$};
				\draw(3*\x,6*\y) node(n17){$s_{1}s_{3}s_{2}\vert s_{3}s_{2}$};
				\draw(.5*\x,7*\y) node(n18){$s_{1}s_{3}s_{2}\vert s_{1}s_{2}\vert s_{1}$};
				\draw(2*\x,7*\y) node(n19){$s_{1}s_{3}s_{2}\vert s_{1}s_{3}s_{2}$};
				\draw(1*\x,8*\y) node(n20){$s_{1}s_{3}s_{2}\vert s_{1}s_{3}s_{2}\vert s_{1}$};
				\draw(3*\x,8*\y) node(n21){$s_{1}s_{3}s_{2}\vert s_{1}s_{3}s_{2}\vert s_{3}$};
				\draw(n1) -- (n2) node[fill=white] at(1.5*\x,1.5*\y) {\tiny\color{gray!80!black}$s_{1}$};
				\draw(n1) -- (n3) node[fill=white] at(2*\x,1.5*\y) {\tiny\color{gray!80!black}$s_{2}$};
				\draw(n1) -- (n4) node[fill=white] at(2.5*\x,1.5*\y) {\tiny\color{gray!80!black}$s_{3}$};
				\draw(n2) -- (n5) node[fill=white] at(1*\x,2.5*\y) {\tiny\color{gray!80!black}$s_{2}$};
				\draw(n2) -- (n6) node[fill=white] at(1.25*\x,2.25*\y) {\tiny\color{gray!80!black}$s_{3}$};
				\draw(n3) -- (n5) node[fill=white] at(1.75*\x,2.25*\y) {\tiny\color{gray!80!black}$s_{1}$};
				\draw(n3) -- (n7) node[fill=white] at(2.25*\x,2.25*\y) {\tiny\color{gray!80!black}$s_{3}$};
				\draw(n4) -- (n6) node[fill=white] at(2.75*\x,2.25*\y) {\tiny\color{gray!80!black}$s_{1}$};
				\draw(n4) -- (n7) node[fill=white] at(3*\x,2.5*\y) {\tiny\color{gray!80!black}$s_{2}$};
				\draw(n5) -- (n8) node[fill=white] at(1*\x,3.5*\y) {\tiny\color{gray!80!black}$s_{1}$};
				\draw(n5) -- (n9) node[fill=white] at(1.5*\x,3.5*\y) {\tiny\color{gray!80!black}$s_{3}$};
				\draw(n6) -- (n9) node[fill=white] at(2*\x,3.5*\y) {\tiny\color{gray!80!black}$s_{2}$};
				\draw(n7) -- (n9) node[fill=white] at(2.5*\x,3.5*\y) {\tiny\color{gray!80!black}$s_{1}$};
				\draw(n7) -- (n10) node[fill=white] at(3*\x,3.5*\y) {\tiny\color{gray!80!black}$s_{3}$};
				\draw(n8) -- (n11) node[fill=white] at(.75*\x,4.5*\y) {\tiny\color{gray!80!black}$s_{2}$};
				\draw(n8) -- (n12) node[fill=white] at(1.25*\x,4.5*\y) {\tiny\color{gray!80!black}$s_{3}$};
				\draw(n9) -- (n12) node[fill=white] at(1.75*\x,4.5*\y) {\tiny\color{gray!80!black}$s_{1}$};
				\draw(n9) -- (n13) node[fill=white] at(2.25*\x,4.5*\y) {\tiny\color{gray!80!black}$s_{3}$};
				\draw(n10) -- (n13) node[fill=white] at(2.75*\x,4.5*\y) {\tiny\color{gray!80!black}$s_{1}$};
				\draw(n10) -- (n14) node[fill=white] at(3.25*\x,4.5*\y) {\tiny\color{gray!80!black}$s_{2}$};
				\draw(n11) -- (n15) node[fill=white] at(.75*\x,5.5*\y) {\tiny\color{gray!80!black}$s_{3}$};
				\draw(n12) -- (n15) node[fill=white] at(1.25*\x,5.5*\y) {\tiny\color{gray!80!black}$s_{2}$};
				\draw(n12) -- (n16) node[fill=white] at(1.75*\x,5.5*\y) {\tiny\color{gray!80!black}$s_{3}$};
				\draw(n13) -- (n16) node[fill=white] at(2.25*\x,5.5*\y) {\tiny\color{gray!80!black}$s_{1}$};
				\draw(n13) -- (n17) node[fill=white] at(2.75*\x,5.5*\y) {\tiny\color{gray!80!black}$s_{2}$};
				\draw(n14) -- (n17) node[fill=white] at(3.25*\x,5.5*\y) {\tiny\color{gray!80!black}$s_{1}$};
				\draw(n15) -- (n18) node[fill=white] at(.75*\x,6.5*\y) {\tiny\color{gray!80!black}$s_{1}$};
				\draw(n15) -- (n19) node[fill=white] at(1.5*\x,6.5*\y) {\tiny\color{gray!80!black}$s_{3}$};
				\draw(n16) -- (n19) node[fill=white] at(2*\x,6.5*\y) {\tiny\color{gray!80!black}$s_{2}$};
				\draw(n17) -- (n19) node[fill=white] at(2.5*\x,6.5*\y) {\tiny\color{gray!80!black}$s_{1}$};
				\draw(n18) -- (n20) node[fill=white] at(.75*\x,7.5*\y) {\tiny\color{gray!80!black}$s_{3}$};
				\draw(n19) -- (n20) node[fill=white] at(1.5*\x,7.5*\y) {\tiny\color{gray!80!black}$s_{1}$};
				\draw(n19) -- (n21) node[fill=white] at(2.5*\x,7.5*\y) {\tiny\color{gray!80!black}$s_{3}$};
			\end{tikzpicture}}
		\caption{Two Bruhat lattices of sortable elements associated with the Coxeter groups $A_{3}$ and $\tilde{C}_{2}$, respectively. Their edges are labeled by the SB-labeling defined in \eqref{eq:sb_bruhat}.}
		\label{fig:bruhat_sortable}
	\end{figure}
\end{example}

\section{Distributivity of the Bruhat Order on Sortable Elements}
	\label{sec:distributivity}
Recall that a lattice $\PP=(P,\leq)$ is \alert{distributive} if it satisfies one of the two following, equivalent, properties for all $p,q,r\in P$:
\begin{align}
	p & \wedge(q\vee r) = (p\wedge q) \vee (p\wedge r)\label{eq:meet_distributivity}\\
	p & \vee(q\wedge r) = (p\vee q) \wedge (p\vee r)\label{eq:join_distributivity}
\end{align}

Armstrong remarked in \cite{armstrong09sorting} that for a certain Coxeter element of the Coxeter group $A_{n}$ the lattice $\bruh$ coincides with the lattice of order ideals of the root poset of $A_{n}$. (For any undefined terminology, we refer once more to \cite{bjorner05combinatorics}.) Hence this particular lattice is distributive. However, Armstrong remarked that this ``phenomenon, unfortunately, does not persist for all types''. In this section we partially answer the question for which finite Coxeter groups and which Coxeter elements the lattice $\bruh$ is distributive. 

\begin{lemma}\label{lem:subgroups_ideals}
	Let $W$ be a Coxeter group, and $\gamma\in W$ a Coxeter element. If $W'$ is a standard parabolic subgroup of $W$ and $\gamma'\in W'$ denotes the restriction of $\gamma$ to $W'$, then $\mathcal{B}_{\gamma'}$ is an order ideal of $\bruh$.
\end{lemma}
\begin{proof}
	Since $W'$ is a subgroup of $W$, every element $w'\in W'$ lies in $W$ as well. Since $W'$ is a standard parabolic subgroup of $W$, we conclude that there exists some $J\subseteq S$ such that $W'$ is generated by $J$. Since $\gamma'$ is the restriction of $\gamma$ to $W'$, we conclude that $\gamma'$ is the subword of $\gamma$ that is obtained by deleting the letters not in $J$. Hence if $w'\in W'$ is $\gamma'$-sortable, then it is also $\gamma$-sortable. It follows immediately if $w\in W$ and $w'\in W'$ satisfy $w\leq_{B}w'$, then we have $w\in W'$. (Otherwise, the $\gamma$-sorting word of $w$ contains a letter not in $J$, which then implies $\alpha_{\gamma}(w)\not\subseteq\alpha_{\gamma}(w')$. This, however, contradicts $w\leq_{B}w'$.)
\end{proof}

\begin{remark}\label{rem:subgroups_intervals}
	In particular, if $W$ is finite, then each standard parabolic subgroup of $W$ induces an interval of $\bruh$.
\end{remark}

Now recall that each Coxeter element $\gamma\in W$ induces an orientation $\Gamma_{\gamma}(W)$ of the Coxeter diagram of $W$ as follows: an edge between $s_{i}$ and $s_{j}$ is oriented 
\raisebox{-.15cm}{\begin{tikzpicture}\small
	\draw(0,0) node(si){$s_{i}$};
	\draw(1,0) node(sj){$s_{j}$};
	\draw(.5,.15) node{\tiny $a$};
	\draw[->,thick](si) -- (.55,0);
	\draw[thick](.55,0) -- (sj);
\end{tikzpicture}}
if and only if $s_{i}$ precedes $s_{j}$ in every reduced word for $\gamma$, see \cite{shi97enumeration}*{Section~1.3}. The next result shows which orientations of a Coxeter diagram induce non-distributive intervals of $\bruh$.

\begin{proposition}\label{prop:forbidden_subgraphs}
	Let $W$ be a Coxeter group, let $\gamma\in W$ be a Coxeter element, and let $\Gamma_{\gamma}(W)$ be the Coxeter diagram of $W$ with the orientation induced by $\gamma$. If $\Gamma_{\gamma}(W)$ contains one of the following induced subgraphs, then $\bruh$ is not distributive:
	\begin{enumerate}[(i)]
		\item \raisebox{-.16cm}{\begin{tikzpicture}\small
				\draw(0,0) node(v1){$s_{i_{1}}$};
				\draw(.75,.2) node{\tiny a};
				\draw(1.5,0) node(v2){$s_{i_{2}}$};
				\draw(2.25,.2) node{\tiny b};
				\draw(3,0) node(v3){$s_{i_{3}}$};
				\draw[thick,->](v2) -- (.7,0);
				\draw[thick](.7,0) -- (v1);
				\draw[thick,->](v2) -- (2.3,0);
				\draw[thick](2.3,0) -- (v3);
			\end{tikzpicture}} for $i_{1},i_{2},i_{3}\in\{1,2,\ldots,n\}$, and $a,b\geq 3$,
		\item \raisebox{-.16cm}{\begin{tikzpicture}\small
				\draw(0,0) node(v1){$s_{i_{1}}$};
				\draw(1.5,0) node(v2){$s_{i_{2}}$};
				\draw(2.25,.2) node{\tiny a};
				\draw(3,0) node(v3){$s_{i_{3}}$};
				\draw[thick,->](v2) -- (.7,0);
				\draw[thick](.7,0) -- (v1);
				\draw[thick,->](v3) -- (2.2,0);
				\draw[thick](2.2,0) -- (v2);
			\end{tikzpicture}} for $i_{1},i_{2},i_{3}\in\{1,2,\ldots,n\}$, and $a\geq 4$,
		\item \raisebox{-.7cm}{\begin{tikzpicture}\small
				\draw(0,0) node(v1){$s_{i_{1}}$};
				\draw(1.5,0) node(v2){$s_{i_{2}}$};
				\draw(3,.5) node(v3){$s_{i_{4}}$};
				\draw(3,-.5) node(v4){$s_{i_{3}}$};
				\draw[thick,->](v1) -- (.8,0);
				\draw[thick](.8,0) -- (v2);
				\draw[thick,->](v3) -- (2.2,.23);
				\draw[thick](2.2,.23) -- (v2);
				\draw[thick,->](v2) -- (2.3,-.26);
				\draw[thick](2.3,-.26) -- (v4);
			\end{tikzpicture}} for $i_{1},i_{2},i_{3},i_{4}\in\{1,2,\ldots,n\}$,
		\item \raisebox{-.7cm}{\begin{tikzpicture}\small
				\draw(0,0) node(v1){$s_{i_{1}}$};
				\draw(1.5,0) node(v2){$s_{i_{2}}$};
				\draw(3,.5) node(v3){$s_{i_{4}}$};
				\draw(3,-.5) node(v4){$s_{i_{3}}$};
				\draw[thick,->](v1) -- (.8,0);
				\draw[thick](.8,0) -- (v2);
				\draw[thick,->](v3) -- (2.2,.23);
				\draw[thick](2.2,.23) -- (v2);
				\draw[thick,->](v4) -- (2.2,-.23);
				\draw[thick](2.2,-.23) -- (v2);
			\end{tikzpicture}} for $i_{1},i_{2},i_{3},i_{4}\in\{1,2,\ldots,n\}$,
		\item \raisebox{-.16cm}{\begin{tikzpicture}\small
				\draw(0,0) node(v1){$s_{i_{1}}$};
				\draw(1,0) node(v2){$s_{i_{2}}$};
				\draw(1.5,.2) node{\tiny a};
				\draw(2,0) node(v3){$s_{i_{3}}$};
				\draw(3,0) node(v4){$s_{i_{4}}$};
				\draw[thick,->](v1) -- (.55,0);
				\draw[thick](.55,0) -- (v2);
				\draw[thick,->](v2) -- (1.55,0);
				\draw[thick](1.55,0) -- (v3);
				\draw[thick,->](v4) -- (2.45,0);
				\draw[thick](2.45,0) -- (v3);
			\end{tikzpicture}} for $i_{1},i_{2},i_{3},i_{4}\in\{1,2,\ldots,n\}$, and $a\geq 4$,
		\item \raisebox{-.16cm}{\begin{tikzpicture}\small
				\draw(0,0) node(v1){$s_{i_{1}}$};
				\draw(1,0) node(v2){$s_{i_{2}}$};
				\draw(2,0) node(v3){$s_{i_{3}}$};
				\draw(2.5,.2) node{\tiny a};
				\draw(3,0) node(v4){$s_{i_{4}}$};
				\draw[thick,->](v1) -- (.55,0);
				\draw[thick](.55,0) -- (v2);
				\draw[thick,->](v2) -- (1.55,0);
				\draw[thick](1.55,0) -- (v3);
				\draw[thick,->](v3) -- (2.55,0);
				\draw[thick](2.55,0) -- (v4);
			\end{tikzpicture}} for $i_{1},i_{2},i_{3},i_{4}\in\{1,2,\ldots,n\}$, and $a\geq 5$, \quad or
		\item \raisebox{-.16cm}{\begin{tikzpicture}\small
				\draw(0,0) node(v1){$s_{i_{1}}$};
				\draw(1,0) node(v2){$s_{i_{2}}$};
				\draw(2,0) node(v3){$s_{i_{3}}$};
				\draw(2.5,.2) node{\tiny a};
				\draw(3,0) node(v4){$s_{i_{4}}$};
				\draw[thick,->](v1) -- (.55,0);
				\draw[thick](.55,0) -- (v2);
				\draw[thick,->](v2) -- (1.55,0);
				\draw[thick](1.55,0) -- (v3);
				\draw[thick,->](v4) -- (2.45,0);
				\draw[thick](2.45,0) -- (v3);
			\end{tikzpicture}} for $i_{1},i_{2},i_{3},i_{4}\in\{1,2,\ldots,n\}$, and $a\geq 5$.
	\end{enumerate}
\end{proposition}
\begin{proof}
	Suppose that $\Gamma_{\gamma}(W)$ contains an induced subgraph of form (i). Then, in particular, we have the $\gamma$-sortable elements $x=s_{i_{2}}s_{i_{1}}\vert s_{i_{2}},y=s_{i_{2}}s_{i_{1}}s_{i_{3}}$, and $z=s_{i_{2}}s_{i_{3}}\vert s_{i_{2}}$. We have
	\begin{align*}
		x\wedge_{B}(y\vee_{B}z) & = x\wedge_{B}s_{i_{2}}s_{i_{1}}s_{i_{3}}\vert s_{i_{2}}\\
		& = s_{i_{2}}s_{i_{1}}\vert s_{i_{2}},\qquad\text{and}\\
		(x\wedge_{B}y)\vee_{B}(x\wedge_{B}z) & = s_{i_{2}}s_{i_{1}} \vee_{B} s_{i_{2}}\\
		& = s_{i_{2}}s_{i_{1}},
	\end{align*}
	which contradicts \eqref{eq:meet_distributivity}. 

	If $\Gamma_{\gamma}(W)$ contains an induced subgraph of form (ii), then consider the elements $x=s_{i_{2}}s_{i_{1}}\vert s_{i_{2}},y=s_{i_{3}}s_{i_{2}}s_{i_{1}}$, and $z=s_{i_{3}}s_{i_{2}}\vert s_{i_{3}}s_{i_{2}}$. We have 
	\begin{align*}
		x\wedge_{B}(y\vee_{B}z) & = x\wedge_{B}s_{i_{3}}s_{i_{2}}s_{i_{1}}\vert s_{i_{3}}s_{i_{2}}\\
		& = s_{i_{2}}s_{i_{1}}\vert s_{i_{2}},\qquad\text{and}\\
		(x\wedge_{B}y)\vee_{B}(x\wedge_{B}z) & = s_{i_{2}}s_{i_{1}}\vee_{B}s_{i_{2}}\\
		& = s_{i_{2}}s_{i_{1}},
	\end{align*}
	which contradicts \eqref{eq:meet_distributivity}.

	If $\Gamma_{\gamma}(W)$ contains an induced subgraph of the form (iii), then consider the elements $x=s_{i_{4}}s_{i_{2}}s_{i_{3}}\vert s_{i_{2}},y=s_{i_{1}}s_{i_{4}}s_{i_{2}}s_{i_{3}}$, and $z=s_{i_{1}}s_{i_{4}}s_{i_{2}}\vert s_{i_{1}}s_{i_{4}}s_{i_{2}}$. We have
	\begin{align*}
		x\wedge_{B}(y\vee_{B}z) & = x\wedge_{B}s_{i_{1}}s_{i_{4}}s_{i_{2}}s_{i_{3}}\vert s_{i_{1}}s_{i_{4}}s_{i_{2}}\\
		& = s_{i_{4}}s_{i_{2}}s_{i_{3}}\vert s_{i_{2}},\qquad\text{and}\\
		(x\wedge_{B}y)\vee_{B}(x\wedge_{B}z) & = s_{i_{4}}s_{i_{2}}s_{i_{3}}\vee_{B}s_{i_{4}}s_{i_{2}}\\
		& = s_{i_{4}}s_{i_{2}}s_{i_{3}},
	\end{align*}
	which contradicts \eqref{eq:meet_distributivity}.

	If $\Gamma_{\gamma}(W)$ contains an induced subgraph of the form (iv), then consider the elements $x=s_{i_{1}}s_{i_{3}}s_{i_{2}}\vert s_{i_{1}}s_{i_{3}}s_{i_{2}}, y=s_{i_{1}}s_{i_{4}}s_{i_{2}}\vert s_{i_{1}}s_{i_{4}}s_{i_{2}}$, and $z=s_{i_{3}}s_{i_{4}}s_{i_{2}}\vert s_{i_{3}}s_{i_{4}}s_{i_{2}}$. We have
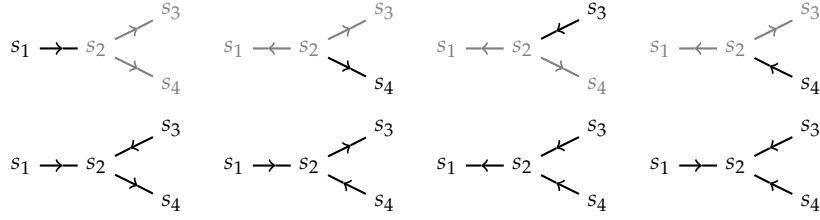
	\begin{align*}
		x\wedge_{B}(y\vee_{B}z) & = x\wedge_{B}s_{i_{1}}s_{i_{3}}s_{i_{4}}s_{i_{2}}\vert s_{i_{1}}s_{i_{3}}s_{i_{4}}s_{i_{2}}\\
		& = s_{i_{1}}s_{i_{3}}s_{i_{2}}\vert s_{i_{1}}s_{i_{3}}s_{i_{2}},\qquad\text{and}\\
		(x\wedge_{B}y)\vee_{B}(x\wedge_{B}z) & = s_{i_{1}}s_{i_{2}}\vert s_{i_{1}}\vee_{B}s_{i_{3}}s_{i_{2}}\vert s_{i_{3}}\\
		& = s_{i_{1}}s_{i_{3}}s_{i_{2}}\vert s_{i_{1}}s_{i_{3}},
	\end{align*}
	which contradicts \eqref{eq:meet_distributivity}.
	
	If $\Gamma_{\gamma}(W)$ contains an induced subgraph of the form (v), then consider the elements $x=s_{i_{2}}s_{i_{4}}s_{i_{3}}\vert s_{i_{2}}s_{i_{4}}s_{i_{3}}\vert s_{i_{2}}, y=s_{i_{4}}s_{i_{3}}\vert s_{i_{4}}$, and $z=s_{i_{1}}s_{i_{2}}s_{i_{3}}\vert s_{i_{1}}s_{i_{2}}s_{i_{3}}\vert s_{i_{1}}s_{i_{2}}$. We have
	\begin{align*}
		x\wedge_{B}(y\vee_{B}z) & = x\wedge_{B}s_{i_{1}}s_{i_{2}}s_{i_{4}}s_{i_{3}}\vert s_{i_{1}}s_{i_{2}}s_{i_{4}}s_{i_{3}}\vert s_{i_{1}}s_{i_{2}}\\
		& = s_{i_{2}}s_{i_{4}}s_{i_{3}}\vert s_{i_{2}}s_{i_{4}}s_{i_{3}}\vert s_{i_{2}},\qquad\text{and}\\
		(x\wedge_{B}y)\vee_{B}(x\wedge_{B}z) & = s_{i_{4}}s_{i_{3}}\vert s_{i_{4}}\vee_{B}s_{i_{2}}s_{i_{3}}\vert s_{i_{2}}s_{i_{3}}\\
		& = s_{i_{2}}s_{i_{4}}s_{i_{3}}\vert s_{i_{2}}s_{i_{4}}s_{i_{3}},
	\end{align*}
	which contradicts \eqref{eq:meet_distributivity}.
	
	If $\Gamma_{\gamma}(W)$ contains an induced subgraph of the form (vi), then consider the elements $x=s_{i_{2}}s_{i_{3}}s_{i_{4}}\vert s_{i_{2}}s_{i_{3}}s_{i_{4}}\vert s_{i_{2}}s_{i_{3}}s_{i_{4}}\vert s_{i_{2}}s_{i_{3}}s_{i_{4}}\vert s_{i_{2}}s_{i_{3}}, y=s_{i_{2}}s_{i_{3}}s_{i_{4}}\vert s_{i_{2}}s_{i_{3}}s_{i_{4}}\vert s_{i_{2}}s_{i_{3}}s_{i_{4}}\vert s_{i_{2}}s_{i_{3}}\vert s_{i_{2}}$, and $z=s_{i_{1}}s_{i_{2}}s_{i_{3}}s_{i_{4}}\vert s_{i_{1}}s_{i_{2}}s_{i_{3}}s_{i_{4}}\vert s_{i_{1}}s_{i_{2}}s_{i_{3}}s_{i_{4}}\vert s_{i_{1}}s_{i_{2}}s_{i_{3}}s_{i_{4}}\vert s_{i_{3}}$. We have
	\begin{align*}
		x\wedge_{B}(y\vee_{B}z) & = x\wedge_{B}s_{i_{1}}s_{i_{2}}s_{i_{3}}s_{i_{4}}\vert s_{i_{1}}s_{i_{2}}s_{i_{3}}s_{i_{4}}\vert s_{i_{1}}s_{i_{2}}s_{i_{3}}s_{i_{4}}\vert s_{i_{1}}s_{i_{2}}s_{i_{3}}s_{i_{4}}\vert s_{i_{2}}s_{i_{3}}\\
		& = s_{i_{2}}s_{i_{3}}s_{i_{4}}\vert s_{i_{2}}s_{i_{3}}s_{i_{4}}\vert s_{i_{2}}s_{i_{3}}s_{i_{4}}\vert s_{i_{2}}s_{i_{3}}s_{i_{4}}\vert s_{i_{2}}s_{i_{3}},\qquad\text{and}\\
		(x\wedge_{B}y)\vee_{B}(x\wedge_{B}z) & = s_{i_{2}}s_{i_{3}}s_{i_{4}}\vert s_{i_{2}}s_{i_{3}}s_{i_{4}}\vert s_{i_{2}}s_{i_{3}}s_{i_{4}}\vert s_{i_{2}}s_{i_{3}}\vert s_{i_{2}}\\
		& \kern1cm \vee_{B}s_{i_{2}}s_{i_{3}}s_{i_{4}}\vert s_{i_{2}}s_{i_{3}}s_{i_{4}}\vert s_{i_{2}}s_{i_{3}}s_{i_{4}}\vert s_{i_{2}}s_{i_{3}}s_{i_{4}}\\
		& = s_{i_{2}}s_{i_{3}}s_{i_{4}}\vert s_{i_{2}}s_{i_{3}}s_{i_{4}}\vert s_{i_{2}}s_{i_{3}}s_{i_{4}}\vert s_{i_{2}}s_{i_{3}}s_{i_{4}}\vert s_{i_{2}},
	\end{align*}
	which contradicts \eqref{eq:meet_distributivity}.
	
	If $\Gamma_{\gamma}(W)$ contains an induced subgraph of the form (vii), then consider the elements $x=s_{i_{2}}s_{i_{4}}s_{i_{3}}\vert s_{i_{2}}s_{i_{4}}s_{i_{3}}\vert s_{i_{2}}s_{i_{4}}s_{i_{3}}\vert s_{i_{2}}s_{i_{4}}s_{i_{3}}\vert s_{i_{2}}s_{i_{4}}s_{i_{3}},y=s_{i_{2}}s_{i_{4}}s_{i_{3}}\vert s_{i_{2}}s_{i_{4}}s_{i_{3}}\vert s_{i_{2}}s_{i_{4}}s_{i_{3}}\vert s_{i_{2}}s_{i_{4}}s_{i_{3}}\vert s_{i_{2}}$, and $z=s_{i_{1}}s_{i_{2}}s_{i_{4}}s_{i_{3}}\vert s_{i_{1}}s_{i_{2}}s_{i_{4}}s_{i_{3}}\vert s_{i_{1}}s_{i_{2}}s_{i_{4}}s_{i_{3}}\vert s_{i_{1}}s_{i_{2}}s_{i_{4}}s_{i_{3}}\vert s_{i_{4}}s_{i_{3}}$. We
	have
	\begin{align*}
		x\wedge_{B}(y\vee_{B}z) & = x\wedge_{B}s_{i_{1}}s_{i_{2}}s_{i_{4}}s_{i_{3}}\vert s_{i_{1}}s_{i_{2}}s_{i_{4}}s_{i_{3}}\vert s_{i_{1}}s_{i_{2}}s_{i_{4}}s_{i_{3}}\vert s_{i_{1}}s_{i_{2}}s_{i_{4}}s_{i_{3}}\vert s_{i_{2}}s_{i_{4}}s_{i_{3}}\\
		& = s_{i_{2}}s_{i_{4}}s_{i_{3}}\vert s_{i_{2}}s_{i_{4}}s_{i_{3}}\vert s_{i_{2}}s_{i_{4}}s_{i_{3}}\vert s_{i_{2}}s_{i_{4}}s_{i_{3}}\vert s_{i_{2}}s_{i_{4}}s_{i_{3}},\qquad\text{and}\\
		(x\wedge_{B}y)\vee_{B}(x\wedge_{B}z) & = s_{i_{2}}s_{i_{4}}s_{i_{3}}\vert s_{i_{2}}s_{i_{4}}s_{i_{3}}\vert s_{i_{2}}s_{i_{4}}s_{i_{3}}\vert s_{i_{2}}s_{i_{4}}s_{i_{3}}\vert s_{i_{2}}\\
		& \kern1cm \vee_{B} s_{i_{2}}s_{i_{4}}s_{i_{3}}\vert s_{i_{2}}s_{i_{4}}s_{i_{3}}\vert s_{i_{2}}s_{i_{4}}s_{i_{3}}\vert s_{i_{2}}s_{i_{4}}s_{i_{3}}\vert s_{i_{4}}\\
		& = s_{i_{2}}s_{i_{4}}s_{i_{3}}\vert s_{i_{2}}s_{i_{4}}s_{i_{3}}\vert s_{i_{2}}s_{i_{4}}s_{i_{3}}\vert s_{i_{2}}s_{i_{4}}s_{i_{3}}\vert s_{i_{2}}s_{i_{4}},
	\end{align*}
	which contradicts \eqref{eq:meet_distributivity}.
\end{proof}
 
We obtain the following corollary immediately.

\begin{corollary}\label{cor:type_d_e}
	If $W=D_{n}$, for $n\geq 4$, $W=E_{n}$, for $n\in\{6,7,8\}$, $W=F_{4}$, or $W=H_{4}$, and $\gamma\in W$ is a Coxeter element, then $\bruh$ is not distributive.
\end{corollary}
\begin{proof}
	First consider $W=D_{4}$. The eight orientations of $\Gamma(D_{4})$ are shown below.
	\begin{center}\begin{tabular}{cccc}
		\begin{tikzpicture}\small
			\draw(0,0) node(v1){$s_{1}$};
			\draw(1,0) node[gray](v2){$s_{2}$};
			\draw(2,.5) node[gray](v3){$s_{3}$};
			\draw(2,-.5) node[gray](v4){$s_{4}$};
			\draw[thick,->](v1) -- (.55,0);
			\draw[thick](.55,0) -- (v2);
			\draw[thick,->,gray](v2) -- (1.55,.27);
			\draw[thick,gray](1.55,.27) -- (v3);
			\draw[thick,->,gray](v2) -- (1.55,-.27);
			\draw[thick,gray](1.55,-.27) -- (v4);
		\end{tikzpicture}
		& 
		\begin{tikzpicture}\small
			\draw(0,0) node[gray](v1){$s_{1}$};
			\draw(1,0) node[gray](v2){$s_{2}$};
			\draw(2,.5) node[gray](v3){$s_{3}$};
			\draw(2,-.5) node(v4){$s_{4}$};
			\draw[thick,->,gray](v2) -- (.45,0);
			\draw[thick,gray](.45,0) -- (v1);
			\draw[thick,->,gray](v2) -- (1.55,.27);
			\draw[thick,gray](1.55,.27) -- (v3);
			\draw[thick,->](v2) -- (1.55,-.27);
			\draw[thick](1.55,-.27) -- (v4);
		\end{tikzpicture}
		& 
		\begin{tikzpicture}\small
			\draw(0,0) node[gray](v1){$s_{1}$};
			\draw(1,0) node[gray](v2){$s_{2}$};
			\draw(2,.5) node(v3){$s_{3}$};
			\draw(2,-.5) node[gray](v4){$s_{4}$};
			\draw[thick,->,gray](v2) -- (.45,0);
			\draw[thick,gray](.45,0) -- (v1);
			\draw[thick,->](v3) -- (1.45,.23);
			\draw[thick](1.45,.23) -- (v2);
			\draw[thick,->,gray](v2) -- (1.55,-.27);
			\draw[thick,gray](1.55,-.27) -- (v4);
		\end{tikzpicture}
		&
		\begin{tikzpicture}\small
			\draw(0,0) node[gray](v1){$s_{1}$};
			\draw(1,0) node[gray](v2){$s_{2}$};
			\draw(2,.5) node[gray](v3){$s_{3}$};
			\draw(2,-.5) node(v4){$s_{4}$};
			\draw[thick,->,gray](v2) -- (.45,0);
			\draw[thick,gray](.45,0) -- (v1);
			\draw[thick,->,gray](v2) -- (1.55,.27);
			\draw[thick,gray](1.55,.27) -- (v3);
			\draw[thick,->](v4) -- (1.45,-.23);
			\draw[thick](1.45,-.23) -- (v2);
		\end{tikzpicture}
		\\
		\begin{tikzpicture}\small
			\draw(0,0) node(v1){$s_{1}$};
			\draw(1,0) node(v2){$s_{2}$};
			\draw(2,.5) node(v3){$s_{3}$};
			\draw(2,-.5) node(v4){$s_{4}$};
			\draw[thick,->](v1) -- (.55,0);
			\draw[thick](.55,0) -- (v2);
			\draw[thick,->](v3) -- (1.45,.23);
			\draw[thick](1.45,.23) -- (v2);
			\draw[thick,->](v2) -- (1.55,-.27);
			\draw[thick](1.55,-.27) -- (v4);
		\end{tikzpicture}
		&
		\begin{tikzpicture}\small
			\draw(0,0) node(v1){$s_{1}$};
			\draw(1,0) node(v2){$s_{2}$};
			\draw(2,.5) node(v3){$s_{3}$};
			\draw(2,-.5) node(v4){$s_{4}$};
			\draw[thick,->](v1) -- (.55,0);
			\draw[thick](.55,0) -- (v2);
			\draw[thick,->](v2) -- (1.55,.27);
			\draw[thick](1.55,.27) -- (v3);
			\draw[thick,->](v4) -- (1.45,-.23);
			\draw[thick](1.45,-.23) -- (v2);
		\end{tikzpicture}
		&
		\begin{tikzpicture}\small
			\draw(0,0) node(v1){$s_{1}$};
			\draw(1,0) node(v2){$s_{2}$};
			\draw(2,.5) node(v3){$s_{3}$};
			\draw(2,-.5) node(v4){$s_{4}$};
			\draw[thick,->](v2) -- (.45,0);
			\draw[thick](.45,0) -- (v1);
			\draw[thick,->](v3) -- (1.45,.23);
			\draw[thick](1.45,.23) -- (v2);
			\draw[thick,->](v4) -- (1.45,-.23);
			\draw[thick](1.45,-.23) -- (v2);
		\end{tikzpicture}
		& 
		\begin{tikzpicture}\small
			\draw(0,0) node(v1){$s_{1}$};
			\draw(1,0) node(v2){$s_{2}$};
			\draw(2,.5) node(v3){$s_{3}$};
			\draw(2,-.5) node(v4){$s_{4}$};
			\draw[thick,->](v1) -- (.55,0);
			\draw[thick](.55,0) -- (v2);
			\draw[thick,->](v3) -- (1.45,.23);
			\draw[thick](1.45,.23) -- (v2);
			\draw[thick,->](v4) -- (1.45,-.23);
			\draw[thick](1.45,-.23) -- (v2);
		\end{tikzpicture}\\
	\end{tabular}\end{center}
	
	The first four orientations in the first row correspond to case (i) in Proposition~\ref{prop:forbidden_subgraphs}, the first three orientations in the second row correspond to case (iii) in Proposition~\ref{prop:forbidden_subgraphs}, and the fourth orientation in the second row corresponds to case (iv) in Proposition~\ref{prop:forbidden_subgraphs}. Hence $\bruh$ cannot be distributive for Coxeter elements inducing these orientations.
	
	If $W=D_{n}$, for $n>4$, or $W=E_{n}$, for $n\in\{6,7,8\}$, then we conclude from Figure~\ref{fig:coxeter_diagrams} that $W$ has a standard parabolic subgroup isomorphic to $D_{4}$. In view of Lemma~\ref{lem:subgroups_ideals} and Remark~\ref{rem:subgroups_intervals}, we conclude that $\bruh$ contains a non-distributive interval, and hence cannot be distributive itself.
	
	Now let $W=F_{4}$. The eight orientations of $\Gamma(F_{4})$ are shown below.
	\begin{center}\begin{tabular}{ccc}
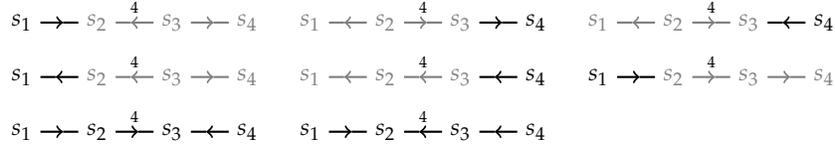
		\begin{tikzpicture}\small
			\draw(0,0) node(v1){$s_{1}$};
			\draw(1,0) node[gray](v2){$s_{2}$};
			\draw(1.5,.2) node{\tiny 4};
			\draw(2,0) node[gray](v3){$s_{3}$};
			\draw(3,0) node[gray](v4){$s_{4}$};
			\draw[thick,->](v1) -- (.55,0);
			\draw[thick](.55,0) -- (v2);
			\draw[thick,->,gray](v3) -- (1.45,0);
			\draw[thick,gray](1.45,0) -- (v2);
			\draw[thick,->,gray](v3) -- (2.55,0);
			\draw[thick,gray](2.55,0) -- (v4);
		\end{tikzpicture}
		&
		\begin{tikzpicture}\small
			\draw(0,0) node[gray](v1){$s_{1}$};
			\draw(1,0) node[gray](v2){$s_{2}$};
			\draw(1.5,.2) node{\tiny 4};
			\draw(2,0) node[gray](v3){$s_{3}$};
			\draw(3,0) node(v4){$s_{4}$};
			\draw[thick,->,gray](v2) -- (.45,0);
			\draw[thick,gray](.45,0) -- (v1);
			\draw[thick,->,gray](v2) -- (1.55,0);
			\draw[thick,gray](1.55,0) -- (v3);
			\draw[thick,->](v3) -- (2.55,0);
			\draw[thick](2.55,0) -- (v4);
		\end{tikzpicture}
		&
		\begin{tikzpicture}\small
			\draw(0,0) node[gray](v1){$s_{1}$};
			\draw(1,0) node[gray](v2){$s_{2}$};
			\draw(1.5,.2) node{\tiny 4};
			\draw(2,0) node[gray](v3){$s_{3}$};
			\draw(3,0) node(v4){$s_{4}$};
			\draw[thick,->,gray](v2) -- (.45,0);
			\draw[thick,gray](.45,0) -- (v1);
			\draw[thick,->,gray](v2) -- (1.55,0);
			\draw[thick,gray](1.55,0) -- (v3);
			\draw[thick,->](v4) -- (2.45,0);
			\draw[thick](2.45,0) -- (v3);
		\end{tikzpicture}
		\\
		\begin{tikzpicture}\small
			\draw(0,0) node(v1){$s_{1}$};
			\draw(1,0) node[gray](v2){$s_{2}$};
			\draw(1.5,.2) node{\tiny 4};
			\draw(2,0) node[gray](v3){$s_{3}$};
			\draw(3,0) node[gray](v4){$s_{4}$};
			\draw[thick,->](v2) -- (.45,0);
			\draw[thick](.45,0) -- (v1);
			\draw[thick,->,gray](v3) -- (1.45,0);
			\draw[thick,gray](1.45,0) -- (v2);
			\draw[thick,->,gray](v3) -- (2.55,0);
			\draw[thick,gray](2.55,0) -- (v4);
		\end{tikzpicture}
		&
		\begin{tikzpicture}\small
			\draw(0,0) node[gray](v1){$s_{1}$};
			\draw(1,0) node[gray](v2){$s_{2}$};
			\draw(1.5,.2) node{\tiny 4};
			\draw(2,0) node[gray](v3){$s_{3}$};
			\draw(3,0) node(v4){$s_{4}$};
			\draw[thick,->,gray](v2) -- (.45,0);
			\draw[thick,gray](.45,0) -- (v1);
			\draw[thick,->,gray](v3) -- (1.45,0);
			\draw[thick,gray](1.45,0) -- (v2);
			\draw[thick,->](v4) -- (2.45,0);
			\draw[thick](2.45,0) -- (v3);
		\end{tikzpicture}
		&
		\begin{tikzpicture}\small
			\draw(0,0) node(v1){$s_{1}$};
			\draw(1,0) node[gray](v2){$s_{2}$};
			\draw(1.5,.2) node{\tiny 4};
			\draw(2,0) node[gray](v3){$s_{3}$};
			\draw(3,0) node[gray](v4){$s_{4}$};
			\draw[thick,->](v1) -- (.55,0);
			\draw[thick](.55,0) -- (v2);
			\draw[thick,->,gray](v2) -- (1.55,0);
			\draw[thick,gray](1.55,0) -- (v3);
			\draw[thick,->,gray](v3) -- (2.55,0);
			\draw[thick,gray](2.55,0) -- (v4);
		\end{tikzpicture}
		\\
		\begin{tikzpicture}\small
			\draw(0,0) node(v1){$s_{1}$};
			\draw(1,0) node(v2){$s_{2}$};
			\draw(1.5,.2) node{\tiny 4};
			\draw(2,0) node(v3){$s_{3}$};
			\draw(3,0) node(v4){$s_{4}$};
			\draw[thick,->](v1) -- (.55,0);
			\draw[thick](.55,0) -- (v2);
			\draw[thick,->](v2) -- (1.55,0);
			\draw[thick](1.55,0) -- (v3);
			\draw[thick,->](v4) -- (2.45,0);
			\draw[thick](2.45,0) -- (v3);
		\end{tikzpicture}
		&
		\begin{tikzpicture}\small
			\draw(0,0) node(v1){$s_{1}$};
			\draw(1,0) node(v2){$s_{2}$};
			\draw(1.5,.2) node{\tiny 4};
			\draw(2,0) node(v3){$s_{3}$};
			\draw(3,0) node(v4){$s_{4}$};
			\draw[thick,->](v1) -- (.55,0);
			\draw[thick](.55,0) -- (v2);
			\draw[thick,->](v3) -- (1.45,0);
			\draw[thick](1.45,0) -- (v2);
			\draw[thick,->](v4) -- (2.45,0);
			\draw[thick](2.45,0) -- (v3);
		\end{tikzpicture}
		& \\
	\end{tabular}\end{center}
	
	The first four orientations correspond to case (i) in Proposition~\ref{prop:forbidden_subgraphs}, the last two orientations in the second row correspond to case (ii) in Proposition~\ref{prop:forbidden_subgraphs}, and the two orientations in the third row correspond to case (v) in Proposition~\ref{prop:forbidden_subgraphs}. Hence $\bruh$ cannot be distributive for Coxeter elements inducing these orientations.
	
	Now let $W=H_{4}$. The eight orientations of $\Gamma(H_{4})$ are shown below.
	\begin{center}\begin{tabular}{ccc}
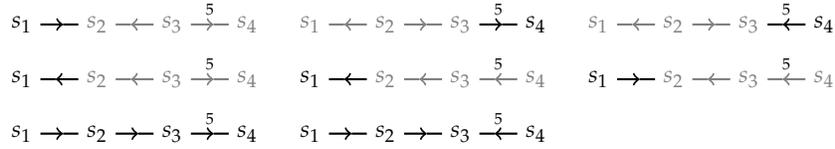
		\begin{tikzpicture}\small
			\draw(0,0) node(v1){$s_{1}$};
			\draw(1,0) node[gray](v2){$s_{2}$};
			\draw(2,0) node[gray](v3){$s_{3}$};
			\draw(2.5,.2) node{\tiny 5};
			\draw(3,0) node[gray](v4){$s_{4}$};
			\draw[thick,->](v1) -- (.55,0);
			\draw[thick](.55,0) -- (v2);
			\draw[thick,->,gray](v3) -- (1.45,0);
			\draw[thick,gray](1.45,0) -- (v2);
			\draw[thick,->,gray](v3) -- (2.55,0);
			\draw[thick,gray](2.55,0) -- (v4);
		\end{tikzpicture}
		&
		\begin{tikzpicture}\small
			\draw(0,0) node[gray](v1){$s_{1}$};
			\draw(1,0) node[gray](v2){$s_{2}$};
			\draw(2,0) node[gray](v3){$s_{3}$};
			\draw(2.5,.2) node{\tiny 5};
			\draw(3,0) node(v4){$s_{4}$};
			\draw[thick,->,gray](v2) -- (.45,0);
			\draw[thick,gray](.45,0) -- (v1);
			\draw[thick,->,gray](v2) -- (1.55,0);
			\draw[thick,gray](1.55,0) -- (v3);
			\draw[thick,->](v3) -- (2.55,0);
			\draw[thick](2.55,0) -- (v4);
		\end{tikzpicture}
		&
		\begin{tikzpicture}\small
			\draw(0,0) node[gray](v1){$s_{1}$};
			\draw(1,0) node[gray](v2){$s_{2}$};
			\draw(2,0) node[gray](v3){$s_{3}$};
			\draw(2.5,.2) node{\tiny 5};
			\draw(3,0) node(v4){$s_{4}$};
			\draw[thick,->,gray](v2) -- (.45,0);
			\draw[thick,gray](.45,0) -- (v1);
			\draw[thick,->,gray](v2) -- (1.55,0);
			\draw[thick,gray](1.55,0) -- (v3);
			\draw[thick,->](v4) -- (2.45,0);
			\draw[thick](2.45,0) -- (v3);
		\end{tikzpicture}
		\\
		\begin{tikzpicture}\small
			\draw(0,0) node(v1){$s_{1}$};
			\draw(1,0) node[gray](v2){$s_{2}$};
			\draw(2,0) node[gray](v3){$s_{3}$};
			\draw(2.5,.2) node{\tiny 5};
			\draw(3,0) node[gray](v4){$s_{4}$};
			\draw[thick,->](v2) -- (.45,0);
			\draw[thick](.45,0) -- (v1);
			\draw[thick,->,gray](v3) -- (1.45,0);
			\draw[thick,gray](1.45,0) -- (v2);
			\draw[thick,->,gray](v3) -- (2.55,0);
			\draw[thick,gray](2.55,0) -- (v4);
		\end{tikzpicture}
		&
		\begin{tikzpicture}\small
			\draw(0,0) node(v1){$s_{1}$};
			\draw(1,0) node[gray](v2){$s_{2}$};
			\draw(2,0) node[gray](v3){$s_{3}$};
			\draw(2.5,.2) node{\tiny 5};
			\draw(3,0) node[gray](v4){$s_{4}$};
			\draw[thick,->](v2) -- (.45,0);
			\draw[thick](.45,0) -- (v1);
			\draw[thick,->,gray](v3) -- (1.45,0);
			\draw[thick,gray](1.45,0) -- (v2);
			\draw[thick,->,gray](v4) -- (2.45,0);
			\draw[thick,gray](2.45,0) -- (v3);
		\end{tikzpicture}
		&
		\begin{tikzpicture}\small
			\draw(0,0) node(v1){$s_{1}$};
			\draw(1,0) node[gray](v2){$s_{2}$};
			\draw(2,0) node[gray](v3){$s_{3}$};
			\draw(2.5,.2) node{\tiny 5};
			\draw(3,0) node[gray](v4){$s_{4}$};
			\draw[thick,->](v1) -- (.55,0);
			\draw[thick](.55,0) -- (v2);
			\draw[thick,->,gray](v3) -- (1.45,0);
			\draw[thick,gray](1.45,0) -- (v2);
			\draw[thick,->,gray](v4) -- (2.45,0);
			\draw[thick,gray](2.45,0) -- (v3);
		\end{tikzpicture}
		\\
		\begin{tikzpicture}\small
			\draw(0,0) node(v1){$s_{1}$};
			\draw(1,0) node(v2){$s_{2}$};
			\draw(2,0) node(v3){$s_{3}$};
			\draw(2.5,.2) node{\tiny 5};
			\draw(3,0) node(v4){$s_{4}$};
			\draw[thick,->](v1) -- (.55,0);
			\draw[thick](.55,0) -- (v2);
			\draw[thick,->](v2) -- (1.55,0);
			\draw[thick](1.55,0) -- (v3);
			\draw[thick,->](v3) -- (2.55,0);
			\draw[thick](2.55,0) -- (v4);
		\end{tikzpicture}
		&
		\begin{tikzpicture}\small
			\draw(0,0) node(v1){$s_{1}$};
			\draw(1,0) node(v2){$s_{2}$};
			\draw(2,0) node(v3){$s_{3}$};
			\draw(2.5,.2) node{\tiny 5};
			\draw(3,0) node(v4){$s_{4}$};
			\draw[thick,->](v1) -- (.55,0);
			\draw[thick](.55,0) -- (v2);
			\draw[thick,->](v2) -- (1.55,0);
			\draw[thick](1.55,0) -- (v3);
			\draw[thick,->](v4) -- (2.45,0);
			\draw[thick](2.45,0) -- (v3);
		\end{tikzpicture}
		& \\
	\end{tabular}\end{center}
	
	The first four orientations correspond to case (i) in Proposition~\ref{prop:forbidden_subgraphs}, and the last two orientations in the second row correspond to case (ii) in Proposition~\ref{prop:forbidden_subgraphs}. The first orientation in the third row corresponds to case (vi) in Proposition~\ref{prop:forbidden_subgraphs}, and the second orientation in the third row corresponds to case (vii) in Proposition~\ref{prop:forbidden_subgraphs}.
\end{proof}

\begin{proposition}\label{prop:distributive_coincidental}
	If $W=A_{n}$ for $n\geq 1$, $W=B_{n}$ for $n\geq 2$, $W=H_{3}$ or $W=I_{2}(k)$ for $k\geq 5$, then there exists a Coxeter element $\gamma\in W$ such that $\bruh$ is distributive.
\end{proposition}
\begin{proof}
	Let $W=A_{n}$ and let $\gamma$ be the Coxeter element that induces the linear orientation 
	\raisebox{-.16cm}{\begin{tikzpicture}\small
			\draw(0,0) node(v1){$s_{1}$};
			\draw(1,0) node(v2){$s_{2}$};
			\draw(2,0) node(v3){$\cdots$};
			\draw(3,0) node(v4){$s_{n}$};
			\draw[thick,->](v1) -- (.55,0);
			\draw[thick](.55,0) -- (v2);
			\draw[thick,->](v2) -- (1.5,0);
			\draw[thick](1.5,0) -- (v3);
			\draw[thick,->](v3) -- (2.6,0);
			\draw[thick](2.6,0) -- (v4);
	\end{tikzpicture}}.
	It follows from the bijection in \cite{bandlow01area} that $\bruh$ is isomorphic to the lattice of classical Dyck paths under dominance order. This lattice is known to be distributive, see for instance \cite{ferrari05lattices}*{Corollary~2.2}. The same bijection implies also that $\bruh$ is isomorphic to the lattice of order ideals of the root poset of $A_{n}$.
	
	Let $W=B_{n}$ and let $\gamma$ be the Coxeter element that induces the linear orientation 
	\raisebox{-.16cm}{\begin{tikzpicture}\small
			\draw(0,0) node(v1){$s_{1}$};
			\draw(1,0) node(v2){$s_{2}$};
			\draw(2,0) node(v3){$\cdots$};
			\draw(3.2,0) node(v4){$s_{n-1}$};
			\draw(3.8,.2) node{\tiny 4};
			\draw(4.2,0) node(v5){$s_{n}$};
			\draw[thick,->](v1) -- (.55,0);
			\draw[thick](.55,0) -- (v2);
			\draw[thick,->](v2) -- (1.5,0);
			\draw[thick](1.5,0) -- (v3);
			\draw[thick,->](v3) -- (2.6,0);
			\draw[thick](2.6,0) -- (v4);
			\draw[thick,->](v4) -- (3.85,0);
			\draw[thick](3.85,0) -- (v5);
	\end{tikzpicture}}.
	It follows from the bijection in \cite{stump13more}*{Section~3} that $\bruh$ is isomorphic to the lattice of type-$B$ Dyck paths under dominance order. This lattice is known to be distributive, see \cite{muehle13heyting}*{Theorem~2.9}. The same bijection implies also that $\bruh$ is isomorphic to the lattice of order ideals of the root poset of $B_{n}$.
	
	Let $W=I_{2}(k)$ for $k\geq 5$, and denote by $s_{1}$ and $s_{2}$ the Coxeter generators of $W$. We have $\mathcal{B}_{s_{1}s_{2}}\cong\mathcal{B}_{s_{2}s_{1}}$, and this lattice is trivially distributive. It is also isomorphic to the lattice of order ideals of the ``root poset'' of $I_{2}(k)$ defined by Armstrong in \cite{armstrong09generalized}*{Figure~5.15}.
	
	Let $W=H_{3}$ and let $\gamma$ be the Coxeter element that induces the linear orientation
	\raisebox{-.16cm}{\begin{tikzpicture}\small
			\draw(0,0) node(v1){$s_{1}$};
			\draw(1,0) node(v2){$s_{2}$};
			\draw(1.5,.2) node{\tiny 5};
			\draw(2,0) node(v4){$s_{3}$};
			\draw[thick,->](v1) -- (.55,0);
			\draw[thick](.55,0) -- (v2);
			\draw[thick,->](v2) -- (1.55,0);
			\draw[thick](1.55,0) -- (v3);
	\end{tikzpicture}}.
	We can easily check by computer that the resulting lattice $\bruh$ is distributive. However, in this case, $\bruh$ is \emph{not} isomorphic to the lattice of order ideals of the ``root poset'' of $H_{3}$ defined by Armstrong in \cite{armstrong09generalized}*{Figure~5.15}.
\end{proof}


\begin{proof}[Proof of Theorem~\ref{thm:main_distributive}]
	This follows immediately from Proposition~\ref{prop:distributive_coincidental}.
\end{proof}

We conclude this section with the following conjecture.

\begin{conjecture}\label{conj:list_exhaustive}
	For finite Coxeter groups, the list in Proposition~\ref{prop:forbidden_subgraphs} is exhaustive, \ie if $W$ is a finite Coxeter group, $\gamma\in W$ is a Coxeter element and the orientation $\Gamma_{\gamma}(W)$ of the Coxeter diagram of $W$ induced by $\gamma$ does not contain one of the induced subgraphs listed in Proposition~\ref{prop:forbidden_subgraphs}, then $\bruh$ is distributive.
\end{conjecture}

\begin{remark}
	The claim of Conjecture~\ref{conj:list_exhaustive} for $W=H_{3}$ can be verified by computer. For $W=B_{3}$, the only orientation other than the one in Proposition~\ref{prop:distributive_coincidental} that is conjectured to yield a distributive lattice $\bruh$ is
	\raisebox{-.16cm}{\begin{tikzpicture}\small
		\draw(0,0) node(v1){$s_{1}$};
		\draw(1,0) node(v2){$s_{2}$};
		\draw(2,0) node(v3){$\cdots$};
		\draw(3.2,0) node(v4){$s_{n-1}$};
		\draw(3.8,.2) node{\tiny 4};
		\draw(4.2,0) node(v5){$s_{n}$};
		\draw[thick,->](v1) -- (.55,0);
		\draw[thick](.55,0) -- (v2);
		\draw[thick,->](v2) -- (1.5,0);
		\draw[thick](1.5,0) -- (v3);
		\draw[thick,->](v3) -- (2.6,0);
		\draw[thick](2.6,0) -- (v4);
		\draw[thick,->](v5) -- (3.75,0);
		\draw[thick](3.75,0) -- (v4);
	\end{tikzpicture}}. For $W=A_{n}$ there are several more options.
\end{remark}

\bibliography{../../literature}

\end{document}